\documentclass[11pt]{amsart}
\usepackage{graphicx}

\usepackage{amsfonts}
\usepackage{amsmath}
\usepackage{amssymb}
\usepackage{amscd}
\usepackage{mathrsfs}  
\usepackage{amsbsy}
\usepackage{amsthm}

\usepackage{CJKutf8}
\usepackage{inputenc}
\usepackage[encapsulated]{CJK}
\usepackage[CJK, overlap]{ruby}

\usepackage[OT2,OT1]{fontenc}  
\newcommand\cyr{%
\renewcommand\rmdefault{wncyr}%
\renewcommand\sfdefault{wncyss}%
\renewcommand\encodingdefault{OT2}%
\normalfont \selectfont} \DeclareTextFontCommand{\textcyr}{\cyr}

\newcommand{\be}{\begin{equation}}
\newcommand{\ee}{\end{equation}}

\newcommand{\bes}{\begin{equation*}}
\newcommand{\ees}{\end{equation*}}

\newcommand{\inn}[2]{{\langle #1,#2 \rangle}}

\newcommand{\C}{\mathbb{C}}

\newcommand{\bH}{\mathbb{H}}

\newcommand{\N}{\mathbb{N}}

\newcommand{\R}{\mathbb{R}}

\newcommand{\mbS}{\mathbb{S}}
\newcommand{\X}{\mathbb{X}}

\newcommand{\Z}{\mathbb{Z}}

\newcommand{\mcD}{\mathcal{D}}

\newcommand{\cA}{\mathcal{A}}

\newcommand{\cO}{\mathcal{O}}

\newcommand{\ba}{\boldsymbol{a}}

\newcommand{\sF}{\mathscr{F}}
\newcommand{\sD}{\mathscr{D}}

\renewcommand{\Im}{\mathop{\mathrm{Im}}}

\newcommand{\bx}{\boldsymbol{x}}

\newcommand{\eps}{\varepsilon}

\newcommand{\Span}{\mathrm{span\,}}

\newcommand{\ml}{\vskip 5pt\noindent}

\newcommand{\wh}[1]{{\widehat{#1}}}
\renewcommand{\rmdefault}{cmr} 
\renewcommand{\sfdefault}{cmr} 

\DeclareMathOperator\loc{{loc}}

\DeclareMathOperator\supp{supp}

\DeclareMathOperator\Sc{Sc}
\DeclareMathOperator\Ve{Vec}
\DeclareMathOperator\re{Re}
\DeclareMathOperator\im{Im}
\DeclareMathOperator\gr{graph}

\def\skpl{\langle}      
\def\skpr{\rangle}      

\newtheorem{theorem}{Theorem}
\theoremstyle{plain}

\newtheorem{corollary}{Corollary}

\newtheorem{definition}{Definition}

\newtheorem{proposition}{Proposition}
\newtheorem{remark}{Remark}

\newtheorem*{notation*}{Notation}
\numberwithin{equation}{section}

\newtheorem{remarks}{Remarks}

\begin{document}

\markboth{Peter Massopust}{Splines and Fractional Differential Operators}

\title{SPLINES AND FRACTIONAL DIFFERENTIAL OPERATORS}

\author{PETER MASSOPUST}

\address{Centre of Mathematics, Technical University of Munich\\ Boltzmannstr. 3,
85478 Garching b. M\"unchen, Germany\\
massopust@ma.tum.de}



\maketitle

\begin{abstract}
Several classes of classical cardinal B-splines can be obtained as solutions of operator equations of the form $Ly = 0$ where $L$ is a linear differential operator of integral order. (Cf., for instance, \cite{akhiezer,Golomb,Krein,micchelli,schoenberg}.) In this article, we  consider classes of generalized B-splines consisting of cardinal polynomial B-splines of complex and hypercomplex orders and cardinal exponential B-splines of complex order, and derive the fractional linear differential operators that are naturally associated with them. For this purpose, we also present the spaces of distributions onto which these fractional differential operators act.
\end{abstract}

\section{Introduction}	

There are several ways of defining and constructing cardinal B-splines: Via recursion relations, convolution products, as a Fourier transform, or as solutions of differential equations of the form $Ly =0$, where $L$ is a linear differential operator of integral order. An albeit incomplete list of references for the latter approach is \cite{akhiezer,Golomb,Krein,micchelli,schoenberg}. In recent years, classes of generalized B-splines were introduced by replacing the integral order in the Fourier representation by a real, complex, or even quaternionic number. (See, for instance, \cite{fub,ub,Hm17,mass14}.)

The reasons for these generalization can be found in the desire to construct approximating functions which close the gaps in the smoothness spectrum of the space $C^n$, $n\in \N$, and which may be used to extract phase information but maintain the computational properties of the classical B-splines of integral order. 

The classical cardinal polynomial and exponential B-splines are also solutions of certain (distributional) differential equations of the form $D^n y = 0$, respectively, $(D^n - a^2)y = 0$, with $D$ denoting the ordinary (distributional) derivative operator and $a\in \R$. In this article, we consider generalized B-splines of non-integral order and relate them to their associated fractional linear differential operators of non-integral order. We also introduce a new class of cardinal polynomial B-splines, namely those of hypercomplex order, and derive their corresponding differential operator. Some of the fractional differential operators mentioned in this article, those associated with polynomial B-splines of complex and quaternionic orders and the exponential B-splines of complex orders, have previously appeared in the literature (cf. \cite{fub,Hm17,mass14}) and their domains identified but we felt the need to provide a coherent setting for their existence and their domains.

The structure of this paper is as follows. Splines and B-splines are briefly introduced in Section 2 and some of their properties mentioned. The next section deals with complex B-splines, their properties, and their associated fractional differential operator defined on Lizorkin spaces. In Section 4, we present exponential splines and their generalization to complex order together with the associated differential operator. The last section is devoted to the new class of hypercomplex B-splines and their corresponding differential operators.

\section{Basics of Spline Functions}
In this section, we briefly review spline functions and review those properties that are important for the remainder of this article. For details and additional information, we refer the interested reader to the vast literature on splines. 

To this end, let $X := \{x_0 < x_1 < \cdots < x_k<x_{k+1}\}$ be a set of knots on the real line $\R$.
\begin{definition}\label{def2.1}
A spline function, or for short a spline, of order $n$ (or degree $n-1$) on $[x_0,x_{k+1}]$ with knot set $X$ is a function $s:[x_0,x_{k+1}]\to\mathbb{R}$ such that\ml
\begin{enumerate}
\item[(i)] On each interval $(x_{i-1},x_i)$, $i = 0,1,\ldots, k+1$, $s$ is a polynomial of order at most $n$;
\item[(ii)]   $s\in C^{n-2}[x_0,x_{k+1}]$.
\end{enumerate}
$s$ is called a {cardinal spline} if the knot set $X$ is a subset of $\Z$.
\end{definition}

\begin{remarks}
\begin{enumerate}
\item In the case $n = 1$, we define $C^{-1}[x_0,x_{k+1}]$ to be the space of piecewise constant functions on $[x_0,x_{k+1}]$.
\item Condition (i) in Definition \ref{def2.1} is equivalent to requiring that $s\vert_{(x_{i-1},x_i)}$ is a solution of the linear differential equation $D^n y = 0$. Here $D := \frac{d}{d x}$ is the ordinary differential operator on functions.
\end{enumerate}
\end{remarks}
As a spline of order $n$ contains $n(k+1)$ free parameters on $[x_0,x_{k+1}]$ and has to satisfy $n-1$ differentiability conditions at the $k$ interior knots $x_1, \ldots, x_k$, the set ${S}_{X,n}$ of all spline functions $s$ of order $n$ over the knot set $X$ forms a real vector space of dimension $n+k$.

\subsection{Schoenberg's Cardinal Polynomial B-Splines}

As the space of cardinal splines of order $n$ over a finite knot set $X$ is finite-dimensional, a convenient and powerful basis for ${S}_{X,n}$ is given by the following family of so-called \emph{(cardinal) $B$-splines}.
\begin{definition}
Denote by $\chi$ the characteristic function of the unit interval $[0,1]$. For $n\in \N$, set
\begin{align}
&B_{1}:[0,1\to [0,1], \quad x\mapsto \chi(x);\\
&B_{n}(x) := (B_{n-1}\ast B_{1})(x) = \int_0^1 B_{n-1} (x - t) dt, \quad 1 < n\in \N.\label{eq2.2}
\end{align}
An element of the discrete family $\{B_n\}_{n\in \N}$ is called a \emph{cardinal polynomial B-spline of order $n$}.
\end{definition}

\begin{remark}
As we exclusively deal with cardinal splines in this article, we will drop the adjective ``cardinal" in the following. 
\end{remark}

Using this definition, it can be shown that $B_n$ has the closed form 
\be
B_n (x)  = \frac{1}{\Gamma (n)}\,\sum_{k=0}^{n} (-1)^k \binom{n}{k} (x-k)_+^{n-1},
\ee
where $\Gamma$ denotes the Euler Gamma function and $x_+^p := \max\{0,x\}^p$ a truncated power function.

Eq. \eqref{eq2.2} implies that the Fourier representation of $B_n$ is given by
\[
\widehat{B}_{n}(\omega) := \sF(B_n)(\omega)  := \int_\R B_n(x) e^{- i\omega x} dx = \left(\frac{1-e^{- i\omega}}{i\omega} \right)^{n}
\]
and that 
\[
B_n\in C^{n-2}, \quad n\in \N.
\]
Furthermore, $\supp B_n = [0,n]$. Hence, $\{B_n : n\in \N\}$ constitutes a \emph{discrete} family of functions with $B_n\in C^{n-2} [0,n]$, $n\in \N$.

It can be established (cf.\,\cite{deboor}) that every cardinal spline function $s:[x_0,x_{k+1}]\to\R$ of order $n$ has a unique representation in terms of a finite sequence of shifted cardinal B-splines of order $n$:
\[
s(x) = \sum\limits_{j=-n+1}^k c_j B_n (x-j),
\]
where $c_j\in \R$. Investigating properties of cardinal splines thus reduces to the investigation of these properties for cardinal B-splines.
\subsection{Polynomial Splines as Solutions of Differential Operators}
Considering splines as solutions of $Ly = 0$, where $L$ is an $n$-th order linear differential operator with constant coefficients goes back to the works of 
Akhiezer, \cite{akhiezer} Krein, \cite{Krein} Golomb, \cite{Golomb} Micchelli, \cite{micchelli} and Schoenberg \cite{schoenberg}. In the latter two references, \emph{cardinal $L$-splines} were defined to be those functions $s\in  C^{n-2} (\R)$ for which $s\vert_{(k,k+1)}$, $k\in \Z$, satisfies
\[
Ly = 0.
\]
In the case of polynomial B-splines, one has
\[
D^{n-2} B_n (x) =  \sum_{k=0}^{n-2} \binom{n-2}{k}(-1)^k B_2(x - k).
\]
As $B_2, B_1\in L^1_{\loc}$, we regard them as regular distributions and thus obtain
\[
D^n B_n (x) = \sum_{k=0}^{n} c_k\, \delta (x - k), \quad c_k\in \R,
\]
where $\delta (\cdot - k)$ denotes the Dirac distribution supported at $x=k$. In the above equation, we now interpret $D$ as a distributional derivative operator on $\mcD := \mcD (\R)$, the space of all infinitely differentiable functions on $\R$ with compact support.

Conversely, every solution of the distributional differential equation
\[
D^n f (x) = \sum_{k=0}^{n} c_k\, \delta (x - k), \quad c_k\in \R
\]
yields a combination of splines of order $n$. (Cf., for instance, \cite{kmps,mass10,sakai1,schu} .)
\section{Complex B-Splines}
Schoenberg's polynomial B-splines have two drawbacks:
\begin{enumerate}
\item[(i)] They form a discrete family of functions of increasing smoothness but there exist functions which belong to spaces of non-integral smoothness. Such functions cannot be well approximated by B-splines of integral order $n$.
\item[(ii)]  The family $\{B_n\}$ interpolates/approximates primarily point values but cannot be used to obtain phase information, i.e., describe both a point value as well as an associated direction. For reasons why phase information and complex-valued transforms are important, we refer the interested reader to \cite{ol,forster14}.
\end{enumerate}
In order to overcome these drawbacks, more general splines were defined recently. First, fractional B-splines, i.e. splines of real order were constructed in \cite{ub} which yield a continuous, with respect to smoothness, family of functions for H\"older spaces, and then B-splines of complex order were defined in\,\cite{fub} which additionally  incorporated phase information. 

B-splines of complex order, for short, complex B-splines, are defined in the Fourier domain as follows:
\be\label{eq3.1}
\widehat{B}_z (\omega) := \left( \frac{1-e^{-i\omega}}{i\omega}\right)^z, \quad \re z > 1.
\ee
It was shown in \cite{fub} that for $\re z > 1$, the expression in parentheses is well-defined as the graph of $\Omega:\R\to\C$, 
\[
\Omega(\omega) := \frac{1-e^{-i\omega}}{i\omega}
\]
does not intersect the negative real axis. A more explicit representation of Eq. \eqref{eq3.1} is given by
\[
\widehat{B}_z(\omega) = \underbrace{\widehat{B}_{\re z}(\omega)}_{\footnotesize{\textrm{continuous smoothness}}}\, \underbrace{e^{i \im z \ln \Omega (\omega)}}_{\footnotesize{\textrm{phase}}} \, \underbrace{e^{- \im z \arg \Omega(\omega)}}_{\footnotesize{\textrm{modulation}}}.
\]
The first term, $\widehat{B}_{\re z}$ defines a fractional spline in the sense of \cite{ub} and yields functions which are elements of the H\"older spaces $C^{\re z -1}$. The second term contains phase information, i.e., defines a direction, and the last term allows for a modulation.

The time domain representation of a complex B-splines $B_z: \R\to \C$ was proved in\,\cite{fub} to be of the form
\be
B_z(x) = \frac{1}{\Gamma(z)} \sum_{k= 0}^\infty (-1)^k \left( {z} \atop {k}\right) (x-k)_+^{z-1},
\ee
where the equality holds point-wise for all $x\in\R$ and in the $L^2(\R)$--norm. Although $\supp B_z = [0,\infty)$, we have 
\[
B_z (x) \in \cO (x^{-m}),\quad\text{for $|x|\to \infty$},
\]
where $m < \re z +1$. (Cf. \cite{fub}.)
\subsection{Lizorkin Spaces}
In order to investigate the type of differential operator that is inherent to complex B-splines, we first need to introduce an appropriate function space. For this purpose, we denote by $\mathcal{S}(\R)$ the Schwartz space of rapidly decreasing functions on $\R$.

\begin{definition} 
The Lizorkin space $\Psi$ is defined by
\[
\Psi := \Psi(\R) := \left\{\psi\in \mathcal{S}(\R) : D^m \psi ({0}) = 0, \,\forall m\in \N\right\},
\]
and its restriction to the nonnegative real axis by
\[
\Psi_+ :=  \Psi(\R_0^+) := \{f\in \Psi : \supp f \subseteq [0,\infty)\}.
\]
\end{definition}
Let $\C_+ := \{z\in \C : \re z > 0\}$ and define a kernel function $K_z:\R\to\C$ by
\[
K_z(x) := \frac{x_+^{z-1}}{\Gamma (z)}.
\]
We can introduce a fractional derivative on $\Psi_+$ as follows. Let $f\in \Psi_+$ and define a fractional derivative operator $\sD^z$ of complex order $z$ on $\Psi_+$ by
\be
\sD^z f := \underbrace{(D^n f) * K_{n-z}}_{\footnotesize{\textrm{(Caputo)}}} = \underbrace{D^n (f*K_{n-z})}_{\footnotesize{\textrm{(Riemann--Liouville)}}}, \qquad n = \lceil \re z \rceil.
\ee
where $*$ denotes the convolution on $\Psi_+$. (Cf. \cite{kst,pod}.) Restricting to the Lizorkin space $\Psi$ allows us to be able not to distinguish between the 
Caputo fractional derivative used in \cite{fub} and the Riemann-Liouville derivative as employed in \cite{statisticalencounters}.

Similarly, a fractional integral operator $\sD^{-z}$ of complex order $z$ on $\Psi_+$ is defined by
\[
\sD^{-z} f := f*K_z.
\]

It follows from the definitions of $\sD^{z}$ and $\sD^{-z}$ that $f\in \Psi_+$ implies $\sD^{\pm z} f \in \Psi_+$ and that all derivatives of $\sD^{\pm z} f$ vanish at $x = 0$. The convolution-based definition of the fractional derivative and integral operator of functions $f\in \Psi_+$ also ensures that $\{\sD^z : z\in \C\}$ is a semi-group in the sense that
\begin{equation}
\sD^{z + \zeta} = \sD^z \sD^\zeta = \sD^\zeta \sD^z = \sD^{\zeta +z},
\end{equation}
for all $z,\zeta\in \C$. (See, for instance, \cite{pod}.)

Next, we extend the above notions to $\Psi_+'$. To this end, we regard $K_z\in L^1_{\loc}$, $\re z > -1$, and thus as an element of $\Psi_+^\prime$ by setting
\[
\skpl K_z, \varphi\skpr = \int_0^\infty K_z (t) \varphi (t) dt, \quad \forall\varphi\in \Psi_+.
\]
Here, $\skpl\cdot, \cdot\skpr$ denotes the canonical pairing between $\Psi_+$ and $\Psi_+^\prime$. 
\begin{remark}
The function $K_z$ may also be defined for general $z\in \C$ via Hadamard's partie finie and represents then a pseudo-function \cite{dantray,gs,zemanian}. 
\end{remark}
%

%
\subsection{The Fractional Derivative on $\Psi_+'$}
For $f,g\in \Psi_+^\prime$ the convolution exists on $\Psi_+^\prime$ and is defined in the usual way \cite{samko} by
\[
\skpl f*g, \varphi\skpr :=  \skpl (f\times g) (x,y), \varphi (x+y)\skpr = \skpl f(x), \skpl g(y), \varphi (x+y)\skpr\skpr, \quad \varphi\in \Psi_+.
\]
Hence, the pair $(\Psi_+^\prime, *)$ is a convolution algebra with the Dirac delta distribution $\delta$ as its unit element. 

Now, we extend the operators $\sD^{\pm z}$ to $\Psi_+^\prime$ as follows. Let  $z\in \C_+$, let $\varphi\in \Psi_+$ be a test function, and let $f\in \Psi_+^\prime$. 
\begin{definition}
The fractional derivative operator $\sD^{z}$ on $\Psi_+^\prime$ is defined by
\begin{align}
\skpl\sD^{z} f, \varphi\skpr &:= \skpl (D^nf)*K_{n-z},\varphi\skpr\nonumber\\
& \phantom{:}= \skpl D^n (f*K_{n-z}),\varphi\skpr, \quad n = \lceil\re z\rceil,\label{eq3.5}
\end{align}
and the fractional integral operator $\sD^{-z}$ by 
\[
\skpl\sD^{-z} f, \varphi\skpr := \skpl f*K_z,\varphi\skpr.
\]
\end{definition}
\noindent
We remark that the semi-group properties carry over to $f\in \Psi_+^\prime$.

Recall that the $z$-th derivative of a truncated power function is given by (cf. \cite{gs})
\[
\sD^z \left[\frac{(x - k)_+^{z-1}}{\Gamma (z)}\right] = \delta (x - k), \quad k < x \in [0, \infty).
\]
The semi-group properties of $\sD^z$ imply that
\[
\sD^{-z} \delta (\cdot - k) = \frac{(\cdot - k)^{z-1}_+}{\Gamma (z)}.
\]
Therefore, we have identified $\sD^{\pm z}$ as the linear differential/integral operators naturally associated with complex B-splines. This gives rise to the following definition.
\begin{definition}
Let $z\in \C_+$ and let $\{a_k : k\in \N_0\}\in \ell^\infty (\R)$. A solution of the equation
\[
\sD^z f = \sum_{k=0}^\infty a_k\, \delta (\cdot - k)
\]
will be termed a \textit{spline of complex order $z$}.
\end{definition}
As the complex B-spline
\[
B_z (x) = \frac{1}{\Gamma(z)} \sum_{k=0}^\infty (-1)^k \binom{z}{k} (x-k)_+^{z-1}, \quad z\in \C_{>1}.
\]
is a solution of  
\[
\sD^z f = \sum_{k=0}^\infty a_k\, \delta (\cdot - k)
\]
with
\[
a_k = (-1)^k \binom{z}{k},
\]
we obtain a nontrivial example of a spline of complex order.

\section{Exponential B-Splines and their Extension to Complex Orders}

Exponential splines are used to model phenomena that follow differential  systems of the form $\dot{x} = A x$, where $A$ is a constant matrix. For 
such equations the solutions are linear combinations of functions of the type $x\mapsto e^{a x}$ and $x\mapsto x^n e^{a x}$, $a\in \R$. 

In  approximation theory, exponential splines are modelling data that exhibit sudden growth or decay and for which polynomials are ill-suited because of their oscillatory behavior \cite{sb}. Analogously to polynomial B-splines, exponential B-splines can be defined as finite convolution products of the exponential functions $e^{a_j (\cdot)}\vert_{[0,1]}$, $a_j \neq 0$. (Cf. \cite{ammar,ChMas,dm1,dm2,mccartin,sakai1, sakai2, spaeth,unserblu05,zoppou} for an albeit incomplete list of references.)

\begin{definition}
Let $n\in \N$ and let $\ba := (a_1, \ldots, a_n)$, where $a_1,\dots, a_n\in \R$ with $a_i\neq 0$ for at least one $i\in \{1, \ldots, n\}$. An exponential B-spline $E_{n,\ba}:[0,n]\to \R$ of order $n$ for the $n$-tuple $\ba$ is a function of the form
\[
{E}_{n} := {E}_{n,\ba} :=  \underset{k = 1}{\overset{n}{*}} e^{a_k(\cdot)}\chi.
\]
\end{definition}
To simplify notation, we set in the following $\eps^{a (\cdot)} := e^{a (\cdot)}\chi$. 

\subsection{Exponential B-Splines and Differential Operators}

We regard $\eps^{a(\cdot)}\in L^1_{\loc}$ as a regular distribution. Then, for any test function $\varphi\in \mathcal{D}$,
\begin{align*}
\inn{D\eps^{a(\cdot)}}{\varphi} & = - \inn{\eps^{a(\cdot)}}{D\varphi} = - \int_0^1 e^{a x} D\varphi (x) dx \\
& = a\left(\varphi (0) - e^a \varphi (1)\right) + a\, \inn{\eps^{a (\cdot)}}{\varphi}\\
& = a \inn{\delta - e^a\, \delta (\cdot - 1)}{\varphi} + a\, \inn{\eps^{a (\cdot)}}{\varphi}\\
\end{align*}
The expression $\nabla^{\exp}_a \,\delta := \delta - e^a\, \delta (\cdot - 1)$ is called the exponential difference operator. Hence,
\[
{(D - a I )\, \eps^{a (\cdot)} = a\, \nabla^{\exp}_a \,\delta}.
\]
In general, using a closed formula for $E_n$\cite{ChMas}, one obtains from the above equation
\[
\prod_{k=1}^n (D - a_k I) E_n = \sum_{k=1}^n b_k\,\delta (\cdot - k), \quad b_k\in \R.
\]
See, also \cite{micchelli,schoenberg}.
\subsection{Fourier Transform of ${E}_n$}
For any $a\in \R,$ 
\[
{\sF}( \eps^{-a(\cdot)})(\omega) = \frac{1- e^{-a} e^{-i \omega}}{i \omega + a}.
\]

Therefore,
\[
\sF({E}_{n,\ba}) (\omega) = \prod_{k=1}^n \frac{1- e^{-a_k} e^{-i \omega}}{i \omega + a_k}\;\;\overset{a_k = a}{=}\;\;\left(\frac{1- e^{-a} e^{-i \omega}}{i \omega + a}\right)^n,
\]
where $\ba := (-a_1, \ldots, -a_n)$. For later purposes, we introduce the function $\Omega_a:\R\to \C$, defined by
\[
\Omega_a (\omega) := \frac{1- e^{-a} e^{-i \omega}}{i \omega + a}.
\]
\subsection{Exponential B-Splines of Complex Order}

Exponential B-spline of complex order were first defined in \cite{mass14}. They combine the advantages of an exponential spline with the phase inherent in complex B-splines. To this end, define $\C_{>1} :=\{z\in \C : \re z > 1\}$.
\begin{definition}
Let $a > 0$ and $z\in \C_{>1}$. An exponential B-spline of complex order, for short, complex exponential B-spline, is defined in the Fourier domain by
\begin{align*}
\widehat{E_z^a} (\omega) &:=\left(\frac{1-e^{-(a+i\omega)}}{a+i\omega}\right)^z\\ \\
& = \widehat{E}_{\re z}^a (\omega)\, e^{i \im z \Omega_a (\omega)}\, e^{- \im z \arg\Omega_a (\omega)}.
\end{align*}
\end{definition}

\begin{remark}
It was shown in \cite{mass14} that $\widehat{E}_z^a$ is well-defined only for $a > 0$.
\end{remark}
The time domain representation of $\widehat{E}_z^a$ was derived in \cite{mass14} using an extension of the exponential difference operator to complex orders. The result is stated in the following theorem whose proof can be found in \cite{mass14}.
\begin{theorem}
Let $z\in \C_{>1}$ and $a > 0$. Then,
\[
E_a^z (x) = \frac{1}{\Gamma(z)}\,\sum_{\ell=0}^\infty \binom{z}{\ell} (-1)^\ell e^{-\ell a} e_+^{-a(x-\ell)}\,(x-\ell)_+^{z-1},
\]
where $e_+^{(\cdot)} := \chi_{[0,\infty)}\,e^{(\cdot)}$. The sum converges point-wise in $\R$ and in the $L^2$--sense.
\end{theorem}

\subsection{A Fractional Differential Operator Associated with $E_a^z$}
As $E_z^a \in L^1_{\loc}$, we can consider it an element of $\Psi_+'$. This, however, implies that $\sD^z E^a_z$ exists where $\sD^z$ is defined as in Eq. \eqref{eq3.5}. In order to obtain an operator inherent to complex exponential B-splines, we need to use the following definition.
\begin{theorem}
Define an operator $(\sD + aI)^z:\Psi_+^\prime\to\Psi_+^\prime$ by
\[
(\sD + aI)^z (e^{-a(\cdot)} f) := e^{-a(\cdot)} \sD^z f.
\]
Then following hold:
\begin{itemize}
\item[(a)] As $f\equiv 1\in \Psi_+^\prime$, the function $e^{-a(\cdot)}\in \ker (\sD + aI)^z$.
\item[(b)] The complex-valued M\"untz polynomials $(\cdot)^{z-1}$ are in $\Psi_+^\prime$ implying that $(\cdot)^{z-1}\,e^{-a(\cdot)}\in \ker (\sD + aI)^z$.
\end{itemize}
Moreover,
\be\label{eq4.1}
(\sD + aI)^z E_z^a = \sum_{\ell=0}^\infty \left[\binom{z}{\ell} (-1)^\ell e^{-\ell a}\right] \delta (\cdot - \ell).
\ee
\end{theorem}
\begin{proof}
See \cite{mass14}.
\end{proof}
Eq. \eqref{eq4.1} shows that $(\sD + aI)^z$, with $a > 0$ and $\re z > 1$, is the natural linear differential operator associated with complex exponential B-splines. Moreover, it gives rise to the following definition.
\begin{definition}
An exponential spline of complex order $z\in \C_{>1}$ corresponding to $a\in \R^+$ is any solution of the fractional differential equation
\[
(\sD + a I)^{z} f = \sum_{\ell=0}^\infty c_\ell \,\delta (\cdot -\ell),
\]
for some $\ell^\infty$-sequence $\{c_\ell : \ell\in \N\}$.
\end{definition}
Note that the complex exponential spline $E^a_z$ is a nontrivial example of an exponential spline of complex order as the coefficients $c_\ell$ are bounded: 
\[
\left\vert \sum_{\ell=0}^\infty c_\ell\right\vert = \left\vert\sum_{\ell=0}^\infty\binom{z}{\ell} (-1)^\ell e^{-\ell a}\right\vert \leq \left\vert\sum_{\ell=0}^\infty\binom{z}{\ell}\right\vert \leq C\, e^{|z-1|},
\]
for some constant $C>0$.

\section{Hypercomplex B-Splines}
Although complex polynomial or exponential B-splines assign to a location $x$ one direction given by $\Im z$, it is sometimes necessary to describes several different independent directions. Certain types of applications such as geophysical data processing, require a multi-channel description. For instance, seismic data has four channels, each associated with a different kind of seismic wave: the so-called P (Compression), S (Shear), L (Love) and R (Rayleigh) waves. Similarly, the color value of a pixel in a color image has three components -- the red, green and blue channels. 

In order to perform the tasks of processing multi-channel signals and data, an appropriate set of analyzing basis functions is required. These basis functions should have the same analytic properties as those enjoyed by splines of complex order but should in addition be able to describe multi-channel structures. In\ \cite{o1,o2} a set of analyzing functions based on wavelets and Clifford-analytic methodologies was introduced in an effort to process four channel seismic data. A multiresolution structure for the construction of wavelets on the plane for the analysis of four-channel signals was outlined in \cite{hm}. 

In \cite{Hm17}, B-splines of quaternionic order were constructed and several of their properties derived and discussed. Our goal in this section is to extend the construction in \cite{Hm17} from quaternions to hypercomplex numbers in the Clifford algebra $C\ell(n)$, $n>2$, and to obtain the associated linear fractional differential operator.

\subsection{A Brief Introduction to Clifford Algebras}
In order to define the afore-mentioned extension to $C\ell(n)$, we need to review the basics of Clifford algebra. For details, we refer the interested reader for example to \cite{bds} or \cite{ghs}. 

To this end, denote by $\{e_1, \ldots, e_n\}$ the canonical basis of the Euclidean vector space $\R^n$. The real Clifford algebra $C\ell(n)$ generated by $\R^n$ is defined by the multiplication rules $e_i e_j + e_j e_i = -2 \delta_{ij}$, $i,j\in \{1,\ldots, n\}$, where $\delta_{ij}$ is the Kronecker symbol. The dimension of $C\ell(n)$ regarded as a real vector space is $2^n$. 

An element $x\in C\ell(n)$ can be represented in the form $x = \sum\limits_{A} x_A e_A$ with $x_A\in \R$ and $\{e_A : A\subseteq \{1, \ldots, n\}\}$, where $e_A := e_{i_1} e_{i_2} \cdots e_{i_m}$, $1\leq i_1 < \cdots < i_m \leq n$, and $e_\emptyset =: e_0 := 1$. A conjugation on Clifford numbers is defined by $\overline{x} := \sum\limits_{A} x_A \overline{e}_A$ where $\overline{e}_A := \overline{e}_{i_m} \cdots \overline{e}_{i_1}$ with $\overline{e}_i := -e_i$ for $i = 1, \ldots, n$, and $\overline{e}_0 := e_0 = 1$. The Clifford norm of the Clifford number $x = \sum\limits_{A} x_A e_A$ is $|x| := \sqrt{\sum\limits_{A} |x_A|^2}$. 

Two prominent examples of Clifford algebras are the complex numbers $\C = C\ell(1)$ and the real quaternions $\bH = C\ell(2)$. The former is a commutative unital algebra whereas the latter is a noncommutative unital algebra.

An important subspace of $C\ell(n)$ is the space of hypercomplex numbers or paravectors. These are Clifford numbers of the form $\Upsilon = x_0 + \sum\limits_{i=1}^n x_i e_i$. The subspace of hypercomplex numbers is denoted by $\cA_{n+1} := \Span_\R\{e_0, e_1, \ldots, e_n\} = \R\oplus \R^n$. Note that each hypercomplex number $\Upsilon$ can be identified with an element $(x_0, x_1, \ldots, x_n) =: (x_0, \bx)\in \R\times \R^n$. For many applications in Clifford theory, one therefore identifies $\cA_{n+1}$ with $\R^{n+1}$.

Each hypercomplex number $\cA_{n+1}\ni \Upsilon = x_0 + \sum\limits_{i=1}^n x_i e_i$ may be decomposed as 
\[
\Upsilon = \Sc \Upsilon + \Ve \Upsilon,
\] 
where $\Sc \Upsilon =: s =: x_0$ is the {\it scalar part} of $\Upsilon$ and $\Ve \Upsilon =: v = \sum\limits_{i=1}^n x_i e_i$ is the {\it vector part} of $\Upsilon$. The {conjugate} $\overline{\Upsilon}$ of the hypercomplex number $\Upsilon = s + v$ is the hypercomplex number $\overline{\Upsilon}= s - v$. The Clifford norm of $\Upsilon\in \cA_{n+1}$ is given by $|\Upsilon| = \sqrt{\Upsilon \overline{\Upsilon}} = \sqrt{s^2 + |v|^2} = \sqrt{s^2+\sum\limits_{i=1}^n v_i^2}$.

For later purposes, we also need the complexification of $C\ell (n)$ and its subspace $\cA_{n+1}$. The former is given by $C\ell_\C(n) = C\ell(n) \otimes_\R \C \cong C\ell(n) \oplus i\, C\ell(n)$, where $i$ is the usual complex unit in $\C$, and similarly the latter by $\cA^\C_{n+1} = \cA_{n+1}\otimes_\R \C \cong \cA_{n+1}\oplus i\,\cA_{n+1}$. Note that $\cA^\C_{n+1} = \Span_\C \{e_0, e_1, \ldots, e_n\}$. 

The conjugate of $\Upsilon\in\cA^\C_{n+1}$ is defined by $\overline{\Upsilon} := \overline{x}_0 - \sum\limits_{j=1^n} \overline{x}_j e_j$ where $\overline{x}_k$ denotes the complex conjugate of the complex number $x_k$, $k=0,1,\ldots, n$. Note that the space $\cA^\C_{n+1}$ is closed under multiplication. 
\subsection{Notation and Preliminaries}
For a complex number $z\in \C$, we define the hypercomplex power $z^\Upsilon$ by
\be\label{eq5.1}
z^\Upsilon := z^{x_0}\left(\cos (|v| \log z) + \frac{v}{|v|}\,\sin(|v|\log z)\right), \quad \Upsilon = x_0 + v\in \cA_{n+1}.
\ee
With this definition, the usual rules of differentiation and indefinite integration of powers also hold in the hypercomplex setting.

As the semigroup property $z^{\Upsilon_1} z^{\Upsilon_2} = z^{\Upsilon_1+\Upsilon_2}$ does not hold in the quaternionic case $n = 3$ (cf. \cite{Hm17}) is does not hold in the hypercomplex case either. 

The exponential function $\exp: \cA^\C_{n+1}\to \cA^\C_{n+1}$ is defined by the usual series $\exp \Upsilon := \sum\limits_{k=0}^\infty \frac{\Upsilon^k}{k!}$. We also write $e^\Upsilon$ for $\exp\Upsilon$. It follows directly from the definition of the hypercomplex exponential function that for $\Upsilon\in \cA_{n+1}$, $\exp\Upsilon$ is bounded above by $\exp |\Upsilon|$ and that for $z\in \C$, $\exp (z\Upsilon)$ is bounded above by $\exp (\sqrt{2}\, |z\Upsilon|)$.

Let $\X\in \{\R, \C, \cA_{n+1}, \cA_{n+1}^\C\}$ and let $p\in [1, \infty]$. We denote by $L^2(\R,\X)$ the Banach space of measurable functions $f:\R\to\X$ for which $\int\limits_\R |f(x)|^p dx < \infty$, where the meaning of $|f(x)|$ depends on $\X$. 

On $L^2(\R, \cA_{n+1}^\C)$, the inner product is given by
\[
\inn{f}{g} := \Sc \left( \int_\R f(x) \overline{g(x)} dx\right) = \int_\R \inn{f(x)}{g(x)} dx
\]
and the Fourier-Plancherel transform on $L^1 (\R, \cA_{n+1}^\C)$ by
\[
\sF f (\omega) := \widehat{f} (\omega) := \int_\R f(x) \exp (- i x \omega) dx.
\]
\subsection{Definition of Hypercomplex B-Splines}
In this section, we define hypercomplex splines and exhibit the associated differential operator. To this end, let $\Upsilon\in \cA_{n+1}$ with $x_0:=\Sc\Upsilon > 1$. We denote the vector part of $\Upsilon$ by $v := \Ve \Upsilon = \sum\limits_{i=1}^n x_i e_i$. 

We define a hypercomplex spline $\wh{B_\Upsilon}:\R\to\cA^\C_{n+1}$ in the Fourier domain by:
\be\label{eq6.1}
\wh{B}_\Upsilon (\omega) := \left(\frac{1-e^{-i \omega}}{i \omega}\right)^\Upsilon,
\ee
where, as before, we set $\Omega (\omega) = \dfrac{1-e^{-i \omega}}{i \omega}$ and notice that
the precise meaning of \eqref{eq6.1} is
\begin{align*}
\wh{B}_\Upsilon (\omega) &= \Omega (\omega)^{x_0} \left(\cos (|v|\log \Omega (\omega))+\frac{v}{|v|}\sin (|v|\log \Omega (\omega))\right)\\
& = \wh{B}_{x_0} \left(\cos (|v|\log \Omega (\omega))+\frac{v}{|v|}\sin (|v|\log \Omega (\omega))\right).
\end{align*}
The first factor in the above product represents a fractional B-spline in the sense of \cite{ub} and the second factor is an element of $\mbS^n$, the $n$-dimensional unit sphere in $\R^{n+1}$. This latter factor is interpreted as representing $n$ distinct directions, i.e., an $n$-dimensional phase factor.

Note that $\Omega$ has a removable singularity at $\omega = 0$ and that $\log \Omega$ is well-defined as $\gr\Omega\subset \C$ does not intersect the negative real axis. (See also \cite{fub}.)

Using solely the properties of the Fourier transform and those of fractional B-splines, we conclude the following statements.
\begin{proposition}
Let $B_\Upsilon$ be a hypercomplex B-spline with $x_0 > 1$. Then
\begin{enumerate}
\item $B_\Upsilon\in \cO (|\omega|^{-x_0})$ as $|\omega|\to\infty$.
\item $B_\Upsilon$ reproduces polynomials up to order $\lceil x_0\rceil$ where $\lceil\cdot\rceil:\R\to\Z$ denotes the ceiling function.
\item $B_\Upsilon\in H^{s,p}(\R, \cA^\C_{n+1})$ for $p\in [1, \infty]$ and $s\in [0, x_0+\frac{1}{p})$.
\end{enumerate}
\end{proposition}
Above, $H^{s,p}(\R, \cA^\C_{n+1})$ denotes the Bessel potential space given by
\[
H^{s,p}(\R, \cA^\C_{n+1}) = \left\{f\in L^p (\R, \cA^\C_{n+1}) : \sF^{-1} [(1+ |\cdot|^2)^{s/2} \sF f]\in L^p (\R, \cA^\C_{n+1}) \right\}.
\]
Before deriving an explicit expression of $B_\Upsilon$, we need to introduce the hypercomplex Gamma function. To this end, let $\Upsilon = x_0 + v\in \cA_{n+1}$ and define
\[
\Gamma (\Upsilon) := \int_0^\infty t^{x_0-1} \cos (|v|\log t) e^{-t} dt +\frac{v}{|v|}\int_0^{\infty} t^{x_0-1} \sin (|v|\log t) e^{-t} dt.
\]
More succinctly, we can write the above definition as
\[
\Gamma (\Upsilon) = \int_0^\infty t^{\Upsilon - 1} e^{-t} dt
\]
using \eqref{eq5.1}.
The next result is a generalization of the binomial series to hypercomplex powers.
\begin{theorem}
Let $z\in \C$ with $|z|\leq 1$ and let $\Upsilon = x_0 + v\in \cA_{n+1}$ with $x_0 > 0$. Then 
\[
(1+z)^\Upsilon = \sum_{n=0}^\infty \binom{\Upsilon}{n} z^n,
\]
where the hypercomplex binomial coefficient is defined in terms of the hypercomplex Gamma function.
\end{theorem}
\begin{proof}
The proof in the quaternionic case $n= 3$ (cf. \cite[Theorem 1]{Hm17}) involves only the modulus and the linearity of the vector part $v$ and can be applied directly to the current setting yielding the statement.
\end{proof}
The next theorem extends \cite[Theorem 2]{Hm17}) to the hypercomplex setting and plays a pivotal role in deriving the fractional derivative operators associated with hypercomplex B-splines.
\begin{theorem}\label{th5.2}
Let $\omega > 0$ and $\Upsilon\in \cA_{n+1}$. Then
\[
\int_0^\infty t^\Upsilon e^{-i t \omega} \frac{dt}{t} = \frac{\Gamma(\Upsilon)}{(i \omega)^\Upsilon}.
\]
\end{theorem}
\begin{proof}
Let $\Upsilon = x_0 + v\in \cA_{n+1}$. Defining projectors $\chi_\pm (v) := \frac12 \left(1 \pm \frac{i v}{|v|}\right)$ as in the proof of \cite[Theorem 2]{Hm17}), one quickly verifies that the projection relations are satisfied. Realizing again that the main arguments in the proof of \cite[Theorem 2]{Hm17}) depend only on the linearity and modulus of $v$, they carry over to the present case and give the stated result.
\end{proof}
\subsection{Integral Operators of Hypercomplex Order}
For the purposes of establishing the fractional differential operator associated with hypercomplex B-splines, we recall the following identity\, \cite{gs}\,.

Let $x>0$ and $z\in{\mathbb C}$ with $\re z >0$. Then
\[
x^{-z}=\frac{1}{\Gamma (z)}\int\limits_0^\infty t^z e^{-tx}\,\frac{dt}{t}. 
\]
Now suppose that ${D}$ is the classical first-order differentiation operator and $\Upsilon = x_0 + v\in \cA_{n+1}$ with $x_0 >0$.
\\
Formally define a fractional integration operator ${\sD}^{-\Upsilon}$ of hypercomplex power $\Upsilon$ by
\be\label{eq5.3}
{\sD}^{-\Upsilon} :=\frac{1}{\Gamma (q)}\int_0^\infty t^\Upsilon e^{-t{D}}\,\frac{dt}{t}.
\ee
Consider the operator 
\[
\exp (-t D) := \sum\limits_{j=0}^\infty\frac{(-t)^j}{j!}{D}^j. 
\]
If $f$ is holomorphic on $\R$, it will agree with its Taylor series centered at $0$. Thus,
\[
\exp (-t D) f(x)=f(x-t) =: T_t f(x),
\]
where $T_t$ denotes the translation operator on $\R$.

%
Let $\varphi\in \mathcal{D}$ and $f\in \mathcal{D}'$ be a distribution. If 
$\inn{\cdot}{\cdot}$ denotes the pairing between distributions and test functions, we have that
\begin{align*}
\inn{{\sF}(e^{-t{D}}f )}{\varphi}&=\inn{e^{-t{D}}f}{{\sF}^{-1}\varphi} 
= \inn{{\sF}(f (\cdot - t))}{\varphi },
\end{align*}
in other words, 
\[
e^{-t{D}} f =T_{t} f.
\]
\ml
We now define the hypercomplex power ${\sD}^\Upsilon$ of the operator $D$ on distributions $f\in \mathcal{D}'$ via the Fourier transform by
\be\label{eq5.4}
{\sF}{\sD}^\Upsilon f := (-i \,\cdot\, )^\Upsilon{\sF} f. 
\ee
The operator ${\sD}^{-\Upsilon}$ as given by \eqref{eq5.3} is seen to invert ${\sD}^\Upsilon$. By the fundamental theorem of calculus, this only needs to be verified for $0<x_0<1$. Indeed, for $\varphi\in \mathcal{D}$ and $f\in \mathcal{D}'$,
\begin{align*}
\inn{{\sD}^{-\Upsilon}{\sD}^\Upsilon f}{\varphi} 
&= \frac{1}{\Gamma (\Upsilon)}\inn{\int_0^\infty t^\Upsilon{\sD}^\Upsilon f (\cdot - t)\, \frac{dt}{t}}{\varphi }\\
&=\frac{1}{\Gamma (\Upsilon)}\int_0^\infty t^\Upsilon \inn{f (\cdot - t)}{(-{\sD})^\Upsilon\varphi }\,\frac{dt}{t}\\
&=\frac{1}{\Gamma (\Upsilon)}\int_0^\infty t^\Upsilon({\sF}\inn{f (\cdot - t))}{{\sF}((-{\sD})^\Upsilon\varphi )}\,\frac{dt}{t}\\
&=\frac{1}{\Gamma (\Upsilon)}\inn{\int_0^\infty t^\Upsilon e^{it\omega}\,\frac{dt}{t} {\sF} f}{(i\omega )^\Upsilon{\sF}\varphi }\\
& =\inn{(-i\omega )^{-\Upsilon}{\sF}f}{ (i\omega )^\Upsilon{\sF}\varphi } 
= \inn{f}{\varphi }.
\end{align*}
\begin{theorem}\label{th5.3}
Let $\Upsilon = x_0 + v\in \cA_{n+1}$ with $x_0 > 1$. The hypercomplex B-spline $\wh{B}_\Upsilon:\R\to\cA^\C_{n+1}$ defined in the Fourier domain by \eqref{eq6.1} has the time domain representation $B_\Upsilon: [0,\infty)\to \cA_{n+1}$, 
\be\label{eq5.4}
B_\Upsilon (t) = \frac{1}{\Gamma (\Upsilon)} \sum_{n=0}^\infty (-1)^n \binom{\Upsilon}{n} (t - n)^{\Upsilon-1}_+,
\ee
where the equality holds in the sense of distributions and in $L^2(\R)$.
\end{theorem}
\begin{proof}
The proof mimics the arguments given to establish the validity of \cite[Theorem 3]{Hm17} and is therefore omitted.
\end{proof}

An immediate consequence of \eqref{eq5.4} and Theorem \ref{th5.2} is the following corollary.
\begin{corollary}\label{cor5.1}
The hypercomplex B-spline $B_\Upsilon$ satisfies the distributional differential equation
\be
\sD^\Upsilon B_\Upsilon = \sum_{n=0}^\infty (-1)^n \binom{\Upsilon}{n} \delta (\cdot -n),
\ee
where $\delta$ denotes the Dirac distribution supported on $0$.
\end{corollary}
\begin{proof}
Use Theorems \ref{th5.2} and \ref{th5.3} as well as the fact that $\sD^\Upsilon$ is the inverse to the hypercomplex fractional integral operator defined in \eqref{eq5.3}. (See also the proof of \cite[Theorem 3]{Hm17}.)
\end{proof}
\subsection{Splines of Hypercomplex Orders}
Corollary \ref{cor5.1} shows that the natural (fractional) differential operator associated with a hypercomplex B-spline is $\sD^\Upsilon$ as defined in \eqref{eq5.4}. This then suggests the following definition.
\begin{definition}
A spline of hypercomplex order $\Upsilon = x_0 + v\in \cA_{n+1}$ with $x_0 > 1$ is any solution of the fractional differential equation
\[
\sD^\Upsilon f = \sum_{n=0}^\infty c_n\,\delta (\cdot - n),
\]
where  $\{c_n : n\in \N\}$ is some bounded sequence in $\cA_{n+1}$.
\end{definition}
The hypercomplex B-spline $B_\Upsilon$ is a nontrivial example of such a hypercomplex spline as
\[
\sD^\Upsilon B_\Upsilon = \sum_{n=0}^\infty (-1)^n \binom{\Upsilon}{n}\,\delta (\cdot - n).
\]

\bibliographystyle{plain}
\bibliography{Splines_and_fractional_differential_operators}

\begin{thebibliography}{10}

\bibitem{akhiezer}
N.~Akhiezer.
\newblock {\"U}ber die beste {A}nn{\"a}herung einer {K}lasse stetiger
  periodischer {F}unktionen.
\newblock {\em Doklady Akad. Nauk SSSR}, 17:455--457, 1937.

\bibitem{ammar}
G.~Ammar, W.~Dayawansa, and C.~Martin.
\newblock Exponential interpolation theory: Theory and numerical algorithms.
\newblock {\em Appl. Math. Comput.}, 41:189--232, 1991.

\bibitem{bds}
F.~Brackx, R.~Delanghe, and F.~Sommen.
\newblock {\em Clifford {A}nalysis}.
\newblock Pitman, Boston-London-Melbourne, 1982.

\bibitem{ChMas}
O.~Christensen and P.~Massopust.
\newblock Exponential {B}-splines and the partition of unity property.
\newblock {\em Adv. Comput. Math.}, 37:301--318, 2012.

\bibitem{dm1}
W.~Dahmen and C.~A. Micchelli.
\newblock On the theory and applications of exponential splines.
\newblock In C.~K. Chui, L.~L. Schumaker, and F.~I. Utreras, editors, {\em
  Topics in Multivariate Approximation}. Academic Press, Boston, 1987.

\bibitem{dm2}
W.~Dahmen and C.~A. Micchelli.
\newblock On multivariate {E}-splines.
\newblock {\em Adv. in Math.}, 76:33--93, 1989.

\bibitem{dantray}
R.~Dantray and J.-L. Lions.
\newblock {\em Mathematical {A}nalysis and {N}umerical {M}ethods for {S}cience
  and {T}echnology, {V}ol. 2}.
\newblock Springer Verlag, Berlin, Germany, 2000.

\bibitem{deboor}
C.~de~Boor.
\newblock {\em A {P}ractical {G}uide to {S}plines}.
\newblock Number~27 in Applied Mathematical Sciences. Springer Verlag, 2001.

\bibitem{forster14}
B.~Forster.
\newblock Five {G}ood {R}easons for {C}omplex-{V}alued {T}ransforms in {I}mage
  {P}rocessing.
\newblock In A.~Zayed and G.~Schmeisser, editors, {\em New {P}erspectives on
  {A}pproximation and {S}ampling {T}heory}, Applied and {N}umerical {H}armonic
  {A}nalysis, pages 359--398. Birkh"auser, Basel, 214.

\bibitem{statisticalencounters}
B.~Forster and P.~Massopust.
\newblock {Statistical encounters with complex B-Splines}.
\newblock {\em Constructive Approximation}, 29(3):325--344, 2009.

\bibitem{fub}
B.~Forster, M.~Unser, and T.~Blu.
\newblock Complex {B}-splines.
\newblock {\em Appl. Comput. Harm. Anal.}, 20:261--282, 2006.

\bibitem{gs}
I.~M. Gel'fand and G.~E. Shilov.
\newblock {\em Generalized Functions, Vol. 1: Properties and Operations}.
\newblock {AMS} Chelsea Publishing, 2016.

\bibitem{Golomb}
M.~Golomb.
\newblock Some extremal problems for differentiable periodic functions in
  ${L}_\infty({R})$.
\newblock MRC Tech. Summ. Report 1050, University of Wisconsin, 1970.

\bibitem{ghs}
K.~G{\"u}rlebeck, K.~Habetha, and W.~Spr{\"o}{\ss}ig.
\newblock {\em Funktionentheorie in der {E}bene und im {R}aum}.
\newblock Birkh{\"a}user, 2006.

\bibitem{Hm17}
J.~Hogan and P.~Massopust.
\newblock Quaternionic {B}-{S}plines.
\newblock {\em J. Approx. Th.}, 224:43--65, 2017.

\bibitem{hm}
J.~Hogan and A.~J. Morris.
\newblock Quaternionic wavelets.
\newblock {\em Numer. Funct. Anal. Optim.}, 33(7--9):1095--1111, 2012.

\bibitem{kmps}
S.~Karlin, Ch. Micchelli, A.~Pinkus, and I.~Schoenberg.
\newblock {\em Studies in {S}pline {F}unctions and {A}pproximation {T}heory}.
\newblock Academic Press, New York, 1976.

\bibitem{kst}
A.~Kilbas, H.~Srivastava, and J.~Trujillo.
\newblock {\em Theory and {A}pplications of {F}ractional {D}ifferential
  {E}quations}.
\newblock Elsevier B. V., Amsterdam, The Netherlands, 2006.

\bibitem{Krein}
M.~Krein.
\newblock Sur quelques points de la th\'eorie de la meilleure approximation des
  fonctions p\'eriodiques.
\newblock {\em Doklady Akad. Nauk SSSR}, 18:245--249, 1938.

\bibitem{mass14}
P.~Massopust.
\newblock Exponential splines of complex order.
\newblock {\em Contemporary Mathematics}, 626:87--105, 2014.

\bibitem{mass10}
Peter Massopust.
\newblock {\em Interpolation and Approximation with Splines and Fractals}.
\newblock Oxford University Press, 2012.

\bibitem{mccartin}
B.~J. McCartin.
\newblock Theory of exponential splines.
\newblock {\em J. Approx. Th.}, 66:1--23, 1991.

\bibitem{micchelli}
Ch. Micchelli.
\newblock Cardinal $\mathcal{L}$-{S}plines.
\newblock In S.~Karlin, Ch. Micchelli, A.~Pinkus, and I.~Schoenberg, editors,
  {\em Studies in {S}pline {F}unctions and {A}pproximation {T}heory}, pages
  203--250. Academic Press, New York, 1976.

\bibitem{o1}
S.~Olhede and G.~Metikas.
\newblock The hyperanalytic wavelet transform.
\newblock Statistics Section Technical Report TR-06-02, Imperial College
  London, UK, 2006.

\bibitem{o2}
S.~Olhede and G.~Metikas.
\newblock The monogenic wavelet transform.
\newblock {\em {IEEE} Trans. Signal Proc.}, 57(9):3426--3441, 2009.

\bibitem{ol}
A.~Oppenheim and J.~Lim.
\newblock The importance of phase in signals.
\newblock {\em IEEE}, 69(5):529--541, 1981.

\bibitem{pod}
I.~Podlubny.
\newblock {\em Fractional {D}ifferential {}Equations}.
\newblock Academic Press, San Diego, 1999.

\bibitem{sakai1}
M.~Sakai and R.~A. Usmani.
\newblock On exponential {B}-splines.
\newblock {\em J. Approx. Th.}, 47:122--131, 1986.

\bibitem{sakai2}
M.~Sakai and R.~A. Usmani.
\newblock A class of simple exponential {B}-splines and their application to
  numerical solution to singular perturbation problems.
\newblock {\em Numer. Math.}, 55:493--500, 1989.

\bibitem{samko}
S.~G. Samko, A.~A. Kilbas, and O.~I. Marichev.
\newblock {\em Fractional Integrals and Derivatives}.
\newblock Gordon and Breach Science Publishers, Minsk, Belarus, 1987.

\bibitem{schoenberg}
I.~J. Schoenberg.
\newblock On {M}icchelli's {T}heory of {C}ardinal $\mathcal{L}$-{S}plines.
\newblock In S.~Karlin, Ch. Micchelli, A.~Pinkus, and I.~Schoenberg, editors,
  {\em Studies in {S}pline {F}unctions and {A}pproximation {T}heory}, pages
  251--276. Academic Press, New York, 1976.

\bibitem{schu}
L.~L. Schumaker.
\newblock {\em Spline Functions: Basic Theory}.
\newblock Krieger Publishing Company, 1993.

\bibitem{spaeth}
H.~Spaeth.
\newblock Exponential spline interpolation.
\newblock {\em Computing}, 4:225--233, 1969.

\bibitem{sb}
J.~Stoer and R.~Bulirsch.
\newblock {\em Introduction to Numerical Analysis}.
\newblock Springer Verlag, New York, 2nd edition, 1993.

\bibitem{ub}
M.~Unser and T.~Blu.
\newblock Fractional {S}plines and {Wavelets}.
\newblock {\em {SIAM} {R}eview}, 42(1):43--67, 2000.

\bibitem{unserblu05}
M.~Unser and T.~Blu.
\newblock Cardinal exponential splines: {P}art {I} -- theory and filtering
  algorithms.
\newblock {\em {IEEE} Trans. Signal Processing}, 53(4):1425--1438, 2005.

\bibitem{zemanian}
A.~H. Zemanian.
\newblock {\em Distribution {T}heory and {T}ransform {A}nalysis -- {A}n
  {I}ntroduction to {G}eneralized {F}unctions, with {A}pplications}.
\newblock Dover Publications, Inc., New York, 1987.

\bibitem{zoppou}
C.~Zoppou, S.~Roberts, and R.~J. Renka.
\newblock Exponential spline interpolation in characteristic based scheme for
  solving the advective--diffusion equation.
\newblock {\em Int. J. Numer. Meth. Fluids}, 33:429--452, 2000.

\end{thebibliography}

\end{document}